\documentclass[12pt, a4]{amsart}
\usepackage{amsmath,amssymb,cite,mathrsfs,graphicx}

\newtheorem{prop}{Proposition}

\newtheorem{obs}{Observation}

\theoremstyle{definition}

\theoremstyle{remark}

\numberwithin{equation}{section}
\author{Kohji Matsumoto}
\address{K. Matsumoto, Graduate School of Mathematics, Nagoya University, Chikusa-ku, Nagoya 464-8602, Japan}
\email{kohjimat@math.nagoya-u.ac.jp}
\thanks{}

\author{Mayumi Sh{\=o}ji}
\address{M. Sh{\=o}ji, Department of Mathematical and Physical Sciences, Japan Women's University, Mejirodai, Bunkyoku, Tokyo 112-8681, Japan}
\email{shoji@fc.jwu.ac.jp}
\thanks{}

\begin{document}

\title[Zeros of the Euler double zeta-function]{Numerical computations on the zeros of the Euler double zeta-function I} 
\date{}

\begin{abstract}
The distribution of the zeros of the Euler double zeta-function
$\zeta_2(s_1,s_2)$, in the case when $s_1=s_2$, is studied numerically.
Some similarity to the distribution of the zeros of Hurwitz zeta-functions
is observed.
\end{abstract}

\maketitle

\section{Introduction}\label{sec-1}

The multiple sum
\begin{align}\label{1-1}
\zeta_r(s_1,\ldots,s_r)=\sum_{1\leq n_1<\cdots<n_r}\frac{1}{n_1^{s_1}\cdots
n_r^{s_r}},
\end{align}
where $r\in\mathbb{N}$ and $s_1,\ldots,s_r$ are complex variables, was introduced
independently by Hoffman \cite{Hof92} and Zagier \cite{Zag94}, and has been studied
extensively in recent decades.   Mathematicians were first interested 
in special values of
\eqref{1-1} at positive integer points.   Then around 2000, the meromorphic
continuation of \eqref{1-1} to the whole space $\mathbb{C}^r$ was established, and
mathematicians started to consider analytic properties of \eqref{1-1}.

The case $r=2$ of \eqref{1-1}, that is
\begin{align}\label{1-2}
\zeta_2(s_1,s_2)=\sum_{n_1=1}^{\infty}\sum_{n_2=1}^{\infty}\frac{1}{n_1^{s_1}
(n_1+n_2)^{s_2}},
\end{align}
was first studied already by Euler.   Therefore \eqref{1-2} is sometimes called
the Euler double sum, or the Euler double zeta-function.   Since this is the
simplest case, it is natural to begin analytic studies with \eqref{1-2}.

The series \eqref{1-2} is absolutely convergent in the region
\begin{align}\label{1-3}
\{(s_1,s_2)\in\mathbb{C}^2\;|\;\Re s_1+\Re s_2>2, \quad \Re s_2>1\},
\end{align}
and can be continued meromorphically to the whole space $\mathbb{C}^2$.
The order estimate of $|\zeta_2(s_1,s_2)|$ outside the region \eqref{1-3}
was discussed in \cite{IshMat03}, \cite{KiuTan06}, \cite{KiuTanZha11}.
Various mean values of $|\zeta_2(s_1,s_2)|$ has recently been discussed by
\cite{MatTsuPre} and \cite{IkeMatNagPre}.

The problem of studying the distribution of zeros of \eqref{1-2} (or more
generally, \eqref{1-1}) was first proposed by Zhao \cite{Zha00}.
The case $r=1$ of \eqref{1-1} is nothing but the classical Riemann
zeta-function $\zeta(s)$.   For the Riemann zeta-function there is the famous
Riemann hypothesis, which predicts that all non-trivial zeros (that is, zeros except 
those on the negative real axis) lie on the line $\Re s=1/2$.   In the case of
$\zeta_2(s_1,s_2)$, however, the analogue of the Riemann hypothesis does not hold.
In fact, consider the case $s_1=s_2(=s)$.   Let $N(\sigma',\sigma'',T;\zeta_2)$ 
denotes the number of zeros of $\zeta_2(s,s)$ (counted with multiplicity) in the 
rectangle $\sigma'<\sigma<\sigma''$, $0<t<T$, where $\sigma=\Re s$ and $t=\Im s$.
Then it is shown in a recent preprint of 
Nakamura and Pankowski \cite{NakPanPre} that for any $\sigma'$ and $\sigma''$
satisfying $1/2<\sigma'<\sigma''<1$, it holds that
\begin{align}\label{1-4}
c_1 T\leq N(\sigma',\sigma'',T;\zeta_2)\leq c_2 T
\end{align}
where $c_1,c_2$ are constants with $0<c_1<c_2$.

The reason of this difference lies, probably, on the fact that $\zeta(s)$ has the
Euler product expansion while $\zeta_2(s_1,s_2)$ does not have.   
This is because the additive
structure $n_1+n_2$ in the denominator on the right-hand side of \eqref{1-2} breaks
multiplicative structure.    This observation suggests that the behaviour of zeros
of $\zeta_2(s_1,s_2)$ may resemble the behaviour of zeros of not $\zeta(s)$, but 
Hurwitz zeta-functions
$$
\zeta(s,\alpha)=\sum_{n=0}^{\infty}\frac{1}{(n+\alpha)^{s}}\qquad(0<\alpha\leq 1),
$$
because $\zeta(s,\alpha)$ also includes an additive structure in its denominator.
In fact, the result analogous to \eqref{1-4} is known for Hurwitz zeta-functions;
let $N(\sigma',\sigma'',T;\alpha)$ denotes the number of zeros of
$\zeta(s,\alpha)$ in the rectangle as above.   Then, at least when $\alpha$ is 
a rational number ($\neq 1,1/2$) (due to Voronin) or a transcendental number
(due to Gonek), it holds that $c_3 T\leq N(\sigma',\sigma'',T;\alpha)\leq c_4 T$
with $0<c_3<c_4$
(see Theorems 4.7, 4.8 and 4.10 in \cite[Chapter 8]{LauGar02}).

The purpose of the present series of papers is to study the behaviour of zeros of
$\zeta_2(s_1,s_2)$ from the viewpoint of numerical computations.
In the present paper, as the first step, we will study the case $s_1=s_2=s$, and
especially we will show that the behaviour of zeros of $\zeta_2(s,s)$ is
indeed similar to the behaviour of zeros of Hurwitz zeta-functions in some sense.

\section{The distribution of zeros off the real axis}\label{sec-2}

In this and the next section we describe the results
on the distribution of zeros of $\zeta_2(s,s)$.
In this case, the simplest way of computations is to use the harmonic product
formula
\begin{align}\label{2-1}
\zeta(s_1)\zeta(s_2)=\zeta_2(s_1,s_2)+\zeta_2(s_2,s_1)+\zeta(s_1+s_2).
\end{align}
Putting $s_1=s_2(=s)$ in \eqref{2-1}, we obtain
\begin{align}\label{2-2}
\zeta_2(s,s)=\frac{1}{2}\left\{\zeta(s)^2-\zeta(2s)\right\},
\end{align}
and hence all computations can be done just by using {\it Mathematica 9.0.1.0}, 
in which
the package for the computations of the values of $\zeta(s)$ is equipped.
However in a forthcoming paper we will study more general situation, when
the values of $s_1$ and $s_2$ are different.    As a preparation to such a study, 
in the present paper we will also develop another method, 
which can be applied to the general case.

Here we explain the theoretical background of our second method.
The basic formula which we use in our computations is the following form of the
Euler-Maclaurin summation formula:
\begin{align}\label{2-3}
\zeta_2(s_1,s_2)&=\frac{\zeta(s_1+s_2-1)}{s_2-1}-\frac{\zeta(s_1+s_2)}{2}\\
&+\sum_{q=1}^l (s_2)_q\frac{B_{q+1}}{(q+1)!}\zeta(s_1+s_2+q)-
\sum_{n_1=1}^{\infty}\frac{\phi_l(n_1,s_2)}{n_1^{s_1}},\notag
\end{align}
where $B_q$ is the $q$th Bernoulli number defined by
$t/(e^t-1)=\sum_{q=0}^{\infty}B_q t^q/q!$, 
$(s)_q=s(s+1)\cdots(s+q-1)$, and
\begin{align}\label{2-4}
&\phi_l(n,s)=\sum_{k=1}^n \frac{1}{k^s}-\\
&\quad\left\{\frac{n^{1-s}-1}{1-s}+\frac{1}{2n^s}
-\sum_{q=1}^l\frac{{s}_q B_{q+1}}{(q+1)! n^{s+q}}+\zeta(s)-\frac{1}{s-1}\right\}.
\notag
\end{align}
It is not difficult to see that $\phi_l(n,s)$ is
actually the usual remainder term of the Euler-Maclaurin formula:
\begin{align}\label{2-5}
\phi_l(n,s)=\frac{(s)_{2k+1}}{(2k+1)!}\int_n^{\infty}B_{2k+1}(x-[x])x^{-s-2k-1}dx,
\end{align}
where $k=l/2$ (if $l$ is even) or $=(l+1)/2$ (if $l$ is odd), $[x]$ is the
fractional part of $x$ and $B_q(x)$ is the $q$th Bernoulli polynomial.

This \eqref{2-3} is formula (3) of Akiyama, Egami and Tanigawa \cite{AkiEgaTan01}.
The last sum on the right-hand side of \eqref{2-3} is absolutely convergent
in the region $\Re(s_1+s_2)>-l$.
They use this formula (and its multiple generalization) to show the meromorphic
continuation of \eqref{1-1}.   Moreover they proved that 
$\zeta_2(s_1,s_2)$ is holomorphic except for the singularities
\begin{align}\label{sing}
s_2=1,\quad s_1+s_2=2,1,0,-2,-4,-6,\ldots 
\end{align}

The details how to calculate the zeros by using formula \eqref{2-3} will be
explained in Section \ref{sec-4}.

In this section we consider the distribution of zeros of $\zeta_2(s,s)$ off the
real axis.   Since $\overline{\zeta_2(s,s)}=\zeta_2(\overline{s},\overline{s})$
(here ``bar'' signifies the complex conjugate), it is enough to consider the
situation in the upper half-plane.
Our numerical result on the distribution of zeros is given in Figure \ref{X1-1}.
Let $D(a,b)=\{s\;|\;a<\sigma<b\}$ for any real numbers $a$ and $b$ with $a<b$.
From Figures \ref{X1-1}--\ref{X6-2} we can observe:

\begin{figure}
\begin{center} \leavevmode
\includegraphics[width=0.42\textwidth]{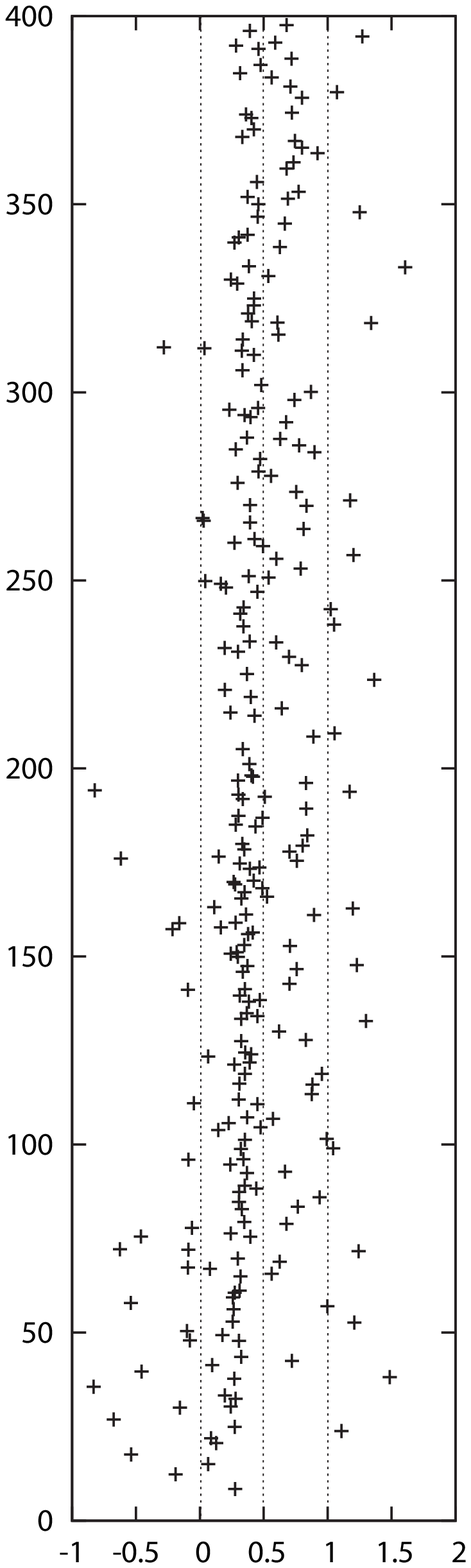} \hskip 1cm
\includegraphics[width=0.42\textwidth]{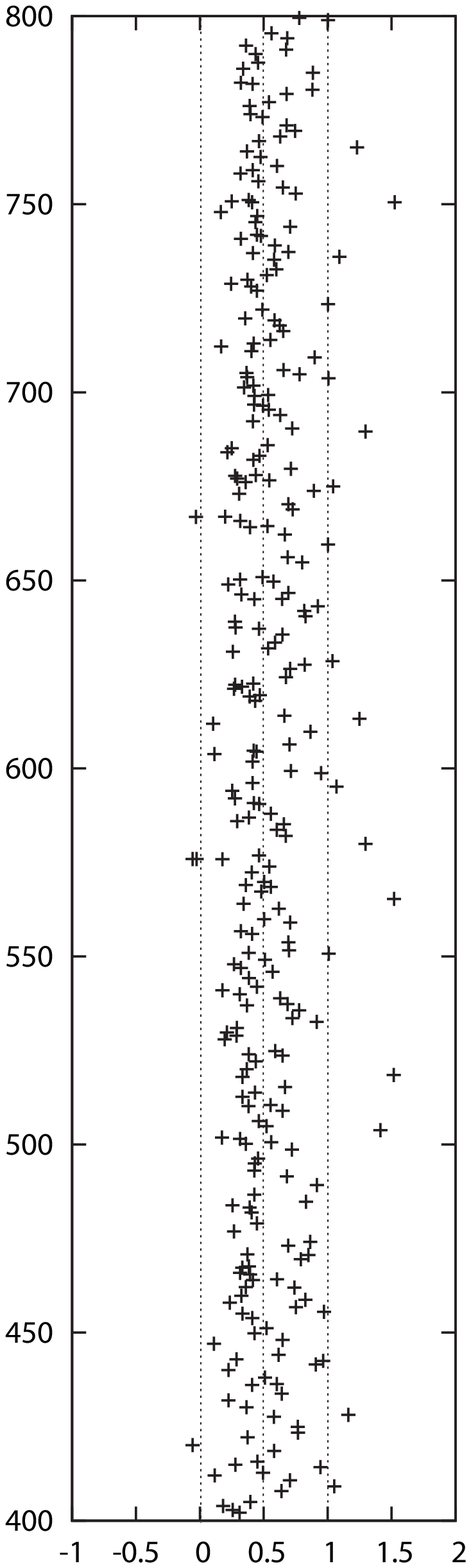}
\end{center}
\caption{The distribution of zeros of $\zeta_2(\sigma+it,\sigma+it)$ for $-1 \leq \sigma \leq 2$ and $0 \leq t \leq 800$. The horizontal axis represents $\sigma$ and the vertical axis does $t$.}
\label{X1-1}
\end{figure}

\begin{figure}
\begin{center} \leavevmode
\includegraphics[width=0.8\textwidth]{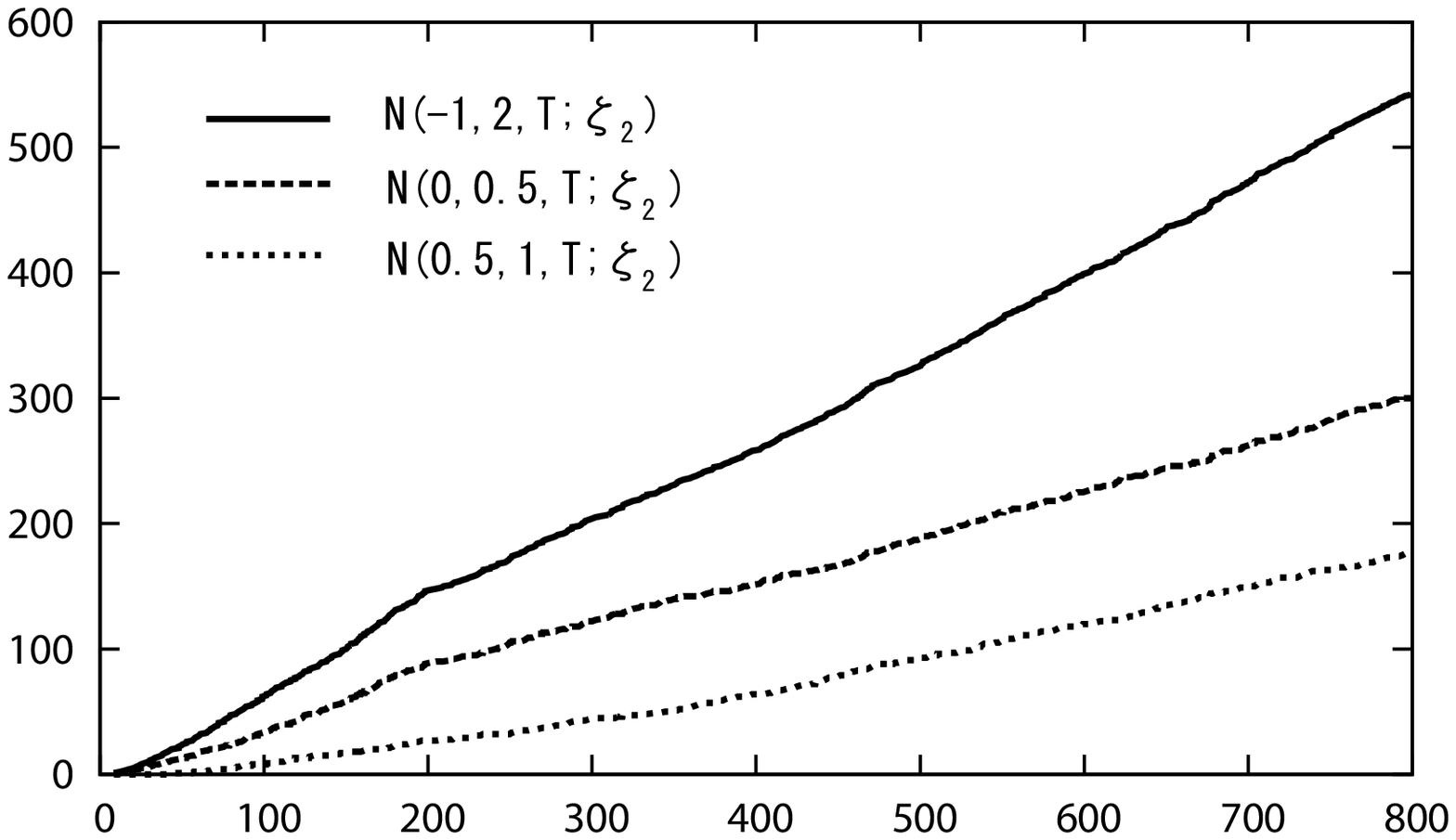}
\end{center}
\caption{Plots of numbers of zeros for $0 \le T \le 800$; $N(-1,2,T;\zeta_2), N(0,1/2,T;\zeta_2)$ and $N(1/2,1,T;\zeta_2)$, respectively from the top.  The horizontal axis represents $T$.}
\label{X1-2}
\end{figure}

\begin{obs}\label{obs1}
{\rm(i)} There are many zeros of $\zeta_2(s,s)$ in the strip $0<\sigma<1$.

{\rm(ii)} It seems that there are more zeros in the region 
$D(0,1/2)$ than in the region $D(1/2,1)$.   {\rm (}More rigorously, it seems that
$N(0,1/2,T;\zeta_2)$ is larger than $N(1/2,1,T;\zeta_2)$.{\rm )}

{\rm(iii)} At least in the range of our computations, there is no zero\footnote{
Some zeros in Figure \ref{X1-1} may seem to be on the line $\sigma=1/2$, but
numerical data shows that they are very close to, but not on that line.
For example there is a zero at $(0.502166\cdots)+i(559.930082\cdots)$.}
lying on the
line $\sigma=1/2$.

{\rm(iv)} There are some zeros in the region $\sigma>1$, but no zero when $\sigma$ is
sufficiently large.

{\rm(v)} There are some zeros in the region $\sigma<0$, but it seems that there is no 
zero when $|\sigma|$ is sufficiently large, and also becomes few and few when $t$ 
becomes large.
\end{obs}

Observation (i) suggests that there should be infinitely many zeros in the region
$D(0,1/2)$ and $D(1/2,1)$.   As mentioned in Section \ref{sec-1}, 
at least in the case 
of $D(1/2,1)$, this fact has already been proved by Nakamura and Pankowski 
\cite{NakPanPre}.   
The graphs in Figure \ref{X1-2} look like straight lines, which suggests that
\begin{align}\label{line}
&N(-1,2,T;\zeta_2)\sim C_1 T, \;\;N(0,1/2,T;\zeta_2)\sim C_2 T,\\
&N(1/2,1,T;\zeta_2)\sim C_3 T\notag
\end{align}
(with some positive constants $C_1, C_2, C_3$) would probably hold, as
$T\to\infty$.
This agrees with \eqref{1-4} of Nakamura and Pankowski.

In Section \ref{sec-1} we also mentioned that the behaviour of zeros of
$\zeta_2(s,s)$ might be similar to that of Hurwitz zeta-functions.   We see that
observations (ii), (iii) and (iv) agree with this expectation.

In fact, in the case of $\zeta(s,\alpha)$, Garunk{\v s}tis and Steuding
\cite[Corollary 3]{GarSte02} proved that there are more zeros of $\zeta(s,\alpha)$
in $D(0,1/2)$ than those in $D(1/2,1)$.   Observation (ii) and Figure \ref{X1-2} 
suggests that the same
situation happens in the case of $\zeta_2(s,s)$. 

The line $\sigma=1/2$ is very important in the theory of $\zeta(s)$, but it seems 
that the same line has no special meaning for Hurwitz zeta-functions.   
Let $N_0(T)$ (resp. $N_0(T,\alpha)$) be the number of zeros in the interval
$\{s\;|\;\sigma=1/2, 0<t<T\}$ of $\zeta(s)$ (resp. $\zeta(s,\alpha)$).
Then it is believed that $N_0(T)\sim (T/2\pi)\log T$, while Gonek \cite{Gon81}
proved that $N_0(T,\alpha)$ is less than $c(T/2\pi)\log T$ with a certain
$c<1$ for $\alpha=1/3,2/3,1/4,3/4,1/6$ and $5/6$.   Moreover in the same paper he
conjectured that $N_0(T,\alpha)\ll T$ for any rational $\alpha\in(0,1)$, 
$\alpha\neq 1/2$. 
Our observation (iii) suggests that the line  $\sigma=1/2$ is also not special for
$\zeta_2(s,s)$.

\begin{figure}
\begin{center} \leavevmode
\includegraphics[width=0.8\textwidth]{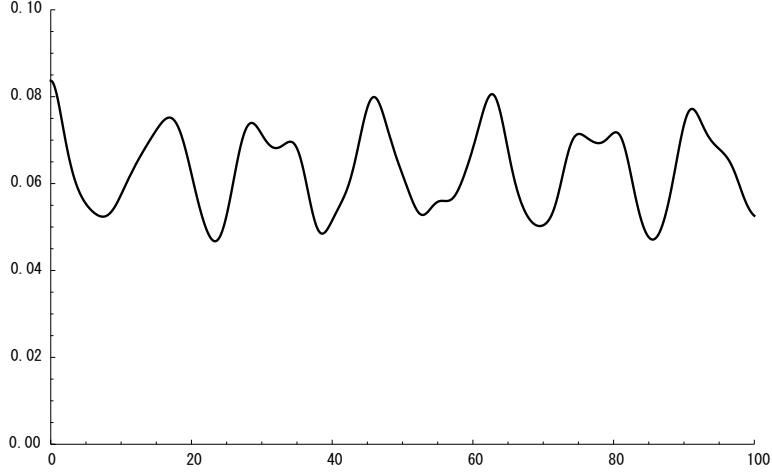} 
\end{center}
\caption{The graph of $|\zeta_2(4+it, 4+it)|$ for $0 \le t \le 100$.}
\label{X6-1}
\end{figure}

\begin{figure}
\begin{center} \leavevmode
\includegraphics[width=0.8\textwidth]{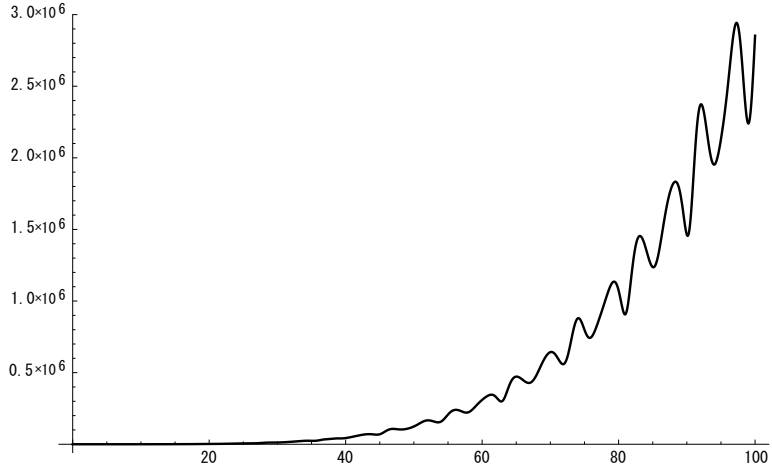}
\end{center}
\caption{The graph of $|\zeta_2(-2+it, -2+it)|$ for $0 \le t \le 100$.}
\label{X6-2}
\end{figure}

The fact corresponding to observation (iv) is classically known for Hurwitz
zeta-functions.   In fact, $\zeta(s,\alpha)\neq 0$ if $\sigma\geq 1+\alpha$
(\cite[Chapter 8, Theorem 1.1]{LauGar02}).
  We note here that the latter part of (iv) can be easily verified
theoretically.

\begin{prop}\label{prop1}
We have $\zeta_2(s,s)\neq 0$ when $\sigma$ is sufficiently large.
\end{prop}

\begin{proof}
Assume $\sigma>1$.   Divide the sum \eqref{1-2} (with $s_1=s_2=s$) as
\begin{align}\label{2-6}
&\zeta_2(s,s)=\sum_{n_2\geq 1}\frac{1}{(1+n_2)^s}
    +\sum_{n_1\geq 2}\frac{1}{n_1^s}\sum_{n_2\geq 1}\frac{1}{(n_1+n_2)^s}\\
&\;=\frac{1}{2^s}+\sum_{n_2\geq 2}\frac{1}{(1+n_2)^s}+\frac{1}{2^s}\sum_{n_2\geq 1}
    \frac{1}{(2+n_2)^s}
    +\sum_{n_1\geq 3}\frac{1}{n_1^s}\sum_{n_2\geq 1}\frac{1}{(n_1+n_2)^s}\notag\\
&\;=\frac{1}{2^s}\left\{1+\sum_{n\geq 3}\frac{1}{(n/2)^s}
    +\sum_{n\geq 3}\frac{1}{n^s} 
    +\sum_{n_1\geq 3}\frac{1}{(n_1/2)^s}\sum_{n_2\geq 1}\frac{1}{(n_1+n_2)^s}
    \right\}.\notag
\end{align}
The second sum on the right-hand side is (putting $n=2k$ when $n$ is even and
$n=2k+1$ when odd) equal to
$$
\sum_{k\geq 2}\frac{1}{k^s}+\sum_{k\geq 1}\frac{1}{(k+1/2)^s}.
$$
We also divide the last sum on the right-hand side of \eqref{2-6} similarly.
Then we obtain
\begin{align}\label{2-7}
\zeta_2(s,s)=\frac{1}{2^s}\{1+Z\},
\end{align}
where
\begin{align}\label{2-8}
Z&=\frac{1}{2^s}+2\sum_{k\geq 3}\frac{1}{k^s}+\sum_{k\geq 1}\frac{1}{(k+1/2)^s}\\
&\;+\sum_{k\geq 2}\frac{1}{k^s}\sum_{n_2\geq 1}\frac{1}{(2k+n_2)^s}
 +\sum_{k\geq 1}\frac{1}{(k+1/2)^s}\sum_{n_2\geq 1}\frac{1}{(2k+1+n_2)^s}.\notag\\
&=\frac{1}{2^s}+2Z_1+Z_2+Z_3+Z_4,\notag
\end{align}
say.   Since
$$
\left|\sum_{n_2\geq 1}\frac{1}{(2k+n_2)^s}\right|\leq
\int_{2k}^{\infty}\frac{dx}{x^{\sigma}}=\frac{(2k)^{1-\sigma}}{\sigma-1},
$$
we have
\begin{align*}
&|Z_3|\leq \frac{2^{1-\sigma}}{\sigma-1}\sum_{k\geq 2}k^{1-2\sigma}
=\frac{2^{1-\sigma}}{\sigma-1}\left(2^{1-2\sigma}+\sum_{k\geq 3}k^{1-2\sigma}
\right)\\
&\leq\frac{2^{1-\sigma}}{\sigma-1}\left(2^{1-2\sigma}+\int_2^{\infty}x^{1-2\sigma}
dx\right)
=\frac{\sigma}{(\sigma-1)^2}2^{2-3\sigma}.
\end{align*}
Similarly we can show
$$
|Z_4|\leq \frac{\sigma-1/4}{(\sigma-1)^2}2^{1-\sigma}\left(\frac{3}{2}\right)
^{1-2\sigma},
$$
and further
$$
|Z_1|\leq \frac{1}{3^{\sigma}}+\int_3^{\infty}x^{-\sigma}dx
=\frac{\sigma+2}{\sigma-1}3^{-\sigma},
$$
$$
|Z_2|\leq \left(\frac{2}{3}\right)^{\sigma}+\int_1^{\infty}(x+1/2)^{-\sigma}dx
=\frac{\sigma+1/2}{\sigma-1}\left(\frac{2}{3}\right)^{\sigma}.
$$
Collecting the above results we find that $Z\to 0$ as $\sigma\to\infty$.
Therefore $1+Z\neq 0$ for sufficiently large $\sigma$, and hence
from \eqref{2-7}
the desired assertion follows.
\end{proof}

\section{Zeros on the real axis}\label{sec-3}

In this section we study the behaviour of $\zeta_2(s,s)$ on the real axis.
When $s_1=s_2=s$, \eqref{2-3} is
\begin{align}\label{3-1}
\zeta_2(s,s)&=\frac{\zeta(2s-1)}{s-1}-\frac{\zeta(2s)}{2}\\            
&+\sum_{q=1}^l (s)_q\frac{B_{q+1}}{(q+1)!}\zeta(2s+q)-                           
\sum_{n_1=1}^{\infty}\frac{\phi_l(n_1,s)}{n_1^{s}}.\notag
\end{align}
From \eqref{3-1} we can observe that $\zeta_2(s,s)$ has a double pole at $s=1$,
a single pole at $s=1/2$, and varies from $+\infty$ to $-\infty$ when $s$ moves from
1 to $1/2$.   Therefore there exists (at least) one zero $s=\sigma_0$ on the interval
$(1/2,1)$.   Figure \ref{X2} shows this situation, and we calculate
$\sigma_0=0.626817\cdots$.

\begin{figure}
\begin{center} \leavevmode
\includegraphics[width=0.45\textwidth]{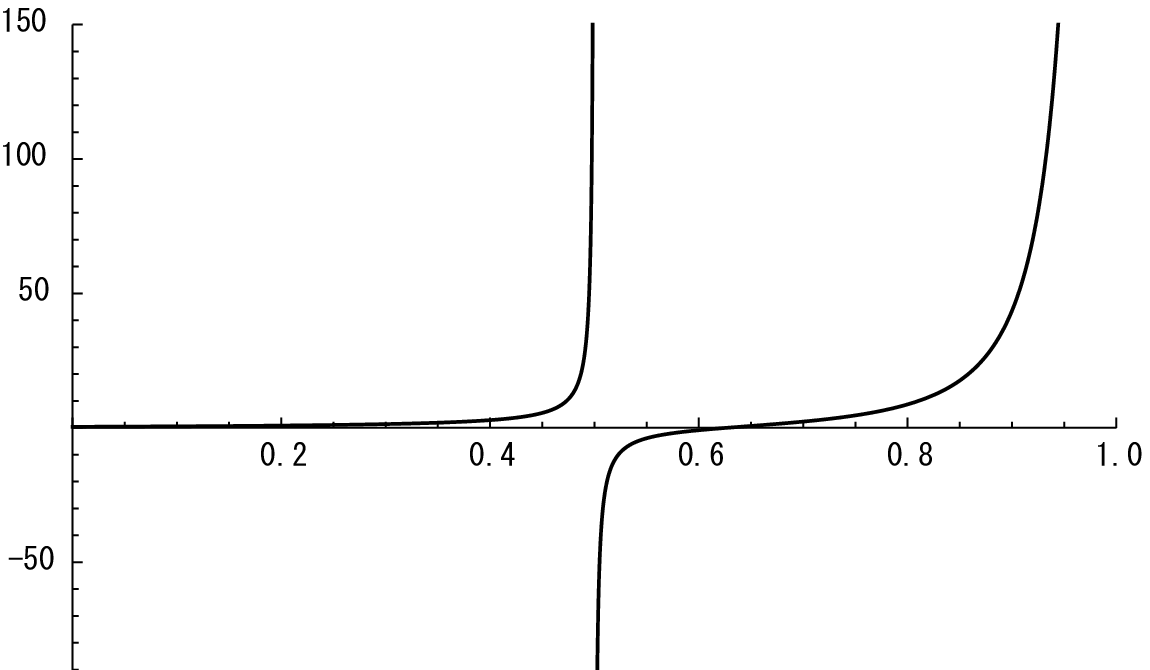} \hskip 1cm
\includegraphics[width=0.45\textwidth]{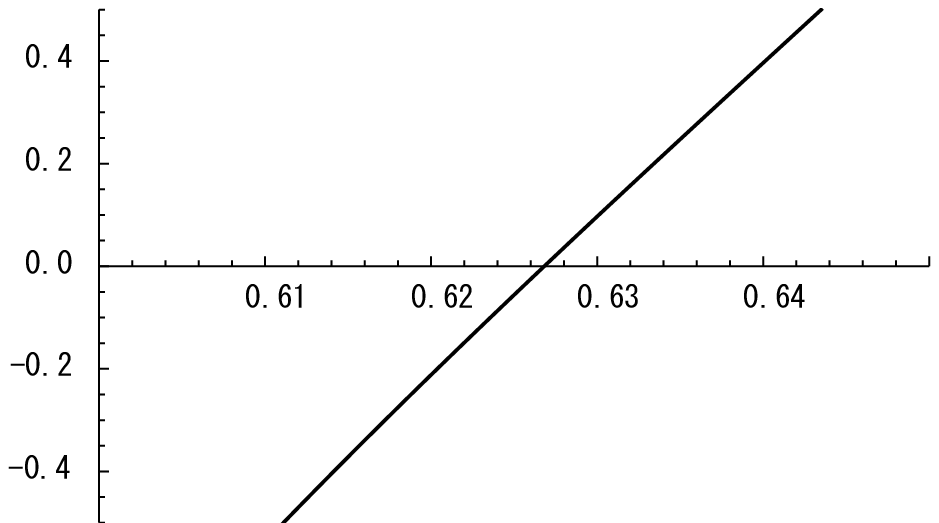}
\end{center}
\caption{The graph of $\zeta_2(\sigma, \sigma)$ for $0 \le \sigma < 1$ (left) and a close-up 
in the vicinity of $\sigma_0=0.626817\cdots$ (right). The vertical axis represents 
$\zeta_2$ and the horizontal axis does $\sigma$.}
\label{X2}
\end{figure}

The list \eqref{sing} of singularities implies that the intersections of the 
hyperplane $s_1=s_2$ and singular loci are
$$
s=1,\frac{1}{2},0,-1,-2,-3,\ldots.
$$
The first two of these points are poles (when restricted to the hyperplane
$s_1=s_2(=s)$), as discussed above.
The points $s=-k$ ($k=0,1,2,\ldots$) are also on singular loci, so they are points 
of indeterminancy.    Figures \ref{X3-1}--\ref{X3-2} are the graph of 
$\zeta_2(\sigma,\sigma)$, therefore
the values at $\sigma=-k$ in Figure \ref{X3-2} show the limit value
\begin{align}\label{3-2}
\lim_{\varepsilon\to 0}\zeta_2(-k+\varepsilon,-k+\varepsilon).
\end{align}
This type of limit is called ``central values'' in Akiyama and Tanigawa
\cite{AkiTan01}.   We use their notation to write \eqref{3-2} as
$\zeta_2^C(-k,-k)$.   Kamano \cite{Kam06} proved an explicit formula for
$\zeta_2^C(-k,-k)$.   Formula (1.6) in \cite{Kam06} implies
\begin{align}\label{3-3}
\zeta_2^C(-k,-k)&=\frac{1}{2}\left\{\zeta(-k)^2-\zeta(-2k)\right\}
\end{align}
for $k\in\mathbb{N}\cup\{0\}$.   In particular, as Kamano stated
\cite[Corollary 2]{Kam06}, $\zeta_2^C(0,0)=3/8$ and $\zeta_2^C(-2k,-2k)=0$ for any 
$k\in\mathbb{N}$.  (The latter was conjectured by Akiyama, Egami and Tanigawa 
\cite{AkiEgaTan01}.)   These values agree with Figure \ref{X3-1} and Figure \ref{X3-2}.

\begin{figure}
\begin{center} \leavevmode
\includegraphics[width=0.45\textwidth]{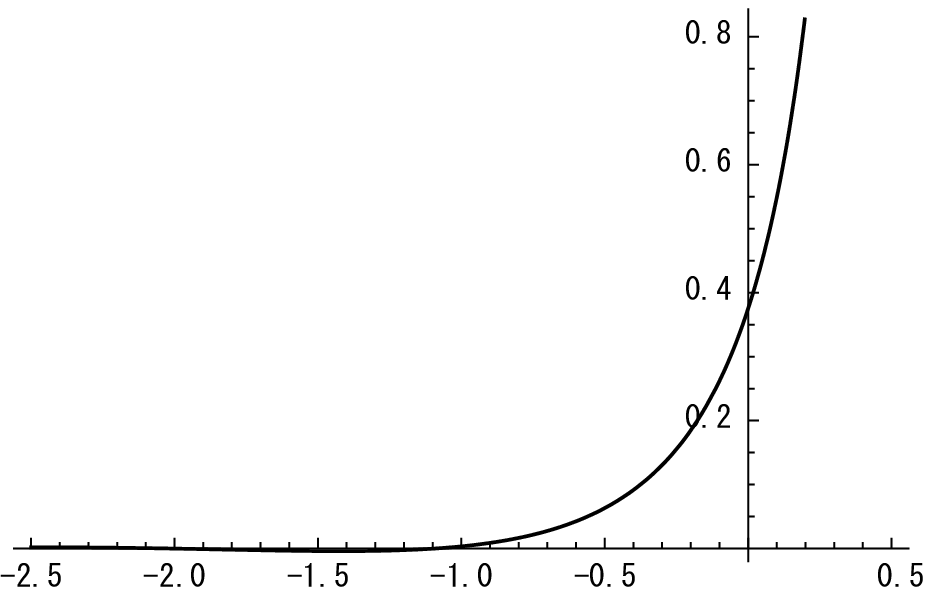} \hskip 1cm
\includegraphics[width=0.45\textwidth]{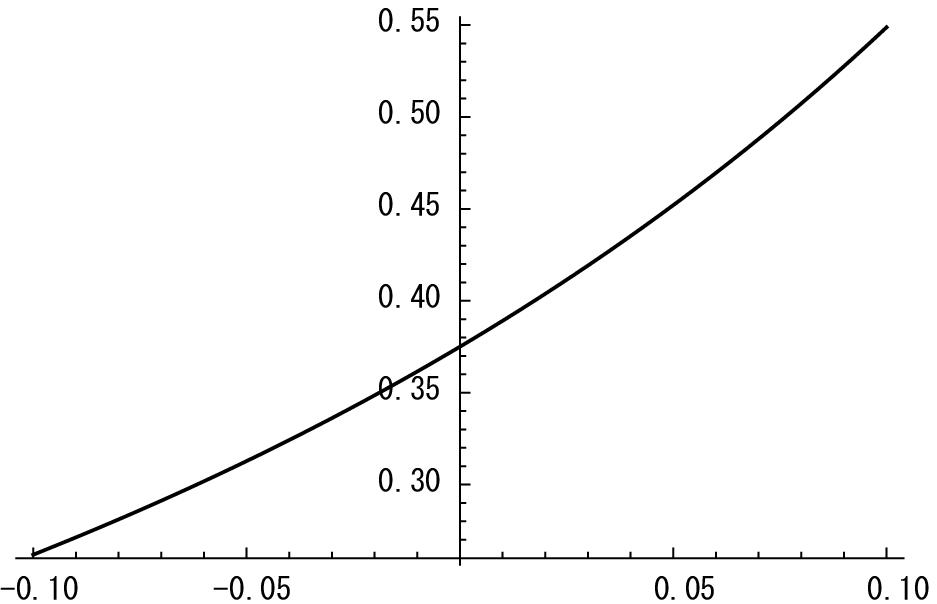}
\end{center}
\caption{The graph of $\zeta_2(\sigma, \sigma)$ for $-2.5 \le \sigma < 0.5$ (left) and a close-up in 
the vicinity of $\zeta_2^C(0,0)$ (right). The vertical axis represents 
$\zeta_2$ and the horizontal axis does $\sigma$.}
\label{X3-1}
\end{figure}

\begin{figure}
\begin{center} \leavevmode
\includegraphics[width=0.6\textwidth]{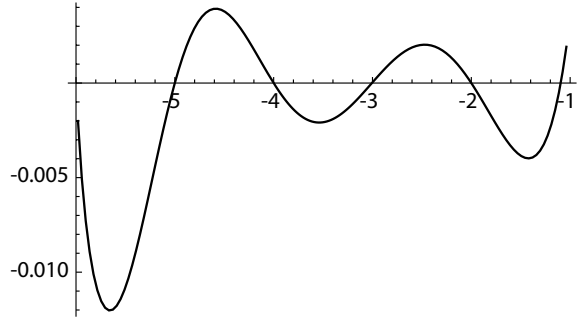}
\end{center}
\caption{The graph of $\zeta_2(\sigma, \sigma)$ for $\sigma < 0$. The vertical axis represents 
$\zeta_2$ and the horizontal axis does $\sigma$.
Amplitude of the vibration becomes intense when $\sigma < -6$. 
Here zeros are $-1.095527\cdots, \; -2, \; -3.005839\cdots, \; -4$ and $-5.000415\cdots$.}
\label{X3-2}
\end{figure}

When $k\in\mathbb{N}$, from \eqref{3-3} we have
\begin{align}\label{3-4}
\zeta_2^C(-k,-k)=\frac{1}{2}\zeta(-k)^2=\frac{(B_{1+k})^2}{2(1+k)^2}.
\end{align}
For example
\begin{align}\label{3-5}
\zeta_2^C(-1,-1)=\frac{1}{288},\;\zeta_2^C(-3,-3)=\frac{1}{28800},
\;\zeta_2^C(-5,-5)=\frac{1}{127008}.
\end{align}
These values are rather small because the corresponding values of Bernoulli
numbers are small ($B_2=1/6$, $B_4=-1/30$, $B_6=1/42$).
This is the reason why $\zeta_2(\sigma,\sigma)$ has zeros near 
$\sigma=-1,-3,-5$ in Figure \ref{X3-2}.   
However, since  $|B_{2k}|\sim 2(2k)!/(2\pi)^{2k}$ as
$k\to\infty$ (\cite[Theorem 12.18]{Apos}), we see that
$\zeta_2^C(-2k+1,-2k+1)\to\infty$ as $k\to\infty$. 

Comparing the above situation on the real axis with Observation \ref{obs1} (v),
we may guess that there would exist some $A>0$, such that $\zeta_2(s,s)$ would have 
no zero in the region $\sigma<-A$ except for the real axis.   This is again
similar to the case of Hurwitz zeta-functions (Theorem 2.7 of
\cite[Chapter 8]{LauGar02}).

\section{The method of computations}\label{sec-4}

In this section we explain the details of our method to compute the zeros of 
$\zeta_2(s,s)$ in $\mathbb{C}$.  As we already mentioned in Section \ref{sec-2},
we apply two methods.   The first method is based on the harmonic product
formula \eqref{2-2}, which gives high-precision zeros of
$\zeta_2(s,s)$ in any digits by virtue of the function {\it Zeta} of {\it Mathematica}.
But we also calculate the zeros by using the Euler-Maclaurin formula.
For this second method, 
instead of the formula \eqref{3-1}, we use the 
following truncated form: 
\begin{align}\label{4-1}
\zeta_2^N(s,s)&:=\frac{\zeta(2s-1)}{s-1}-\frac{\zeta(2s)}{2}\\            
&+\sum_{q=1}^l (s)_q\frac{B_{q+1}}{(q+1)!}\zeta(2s+q)-                           
\sum_{n_1=1}^{N}\frac{\phi_l(n_1,s)}{n_1^{s}},\notag
\end{align}
where the infinite summation of \eqref{3-1} is truncated by $N$ terms.

\subsection{Calculation accuracy}\label{subsec-0}

Obviously, calculation accuracy becomes better as the value of $N$
increases in \eqref{4-1}.  
It is necessary to ascertain the appropriate value of $N$ for our aim.
Fortunately we can get the data of high-precision zeros by \eqref{2-2}.
Therefore, comparing the zeros obtained from \eqref{4-1} with the 
high-precision zeros obtained by \eqref{2-2}, we can estimate the accuracy of zeros 
of \eqref{4-1}.
Appropriate value of $N$ depends on $l$, the number of terms of the first summation.  
As $l$ is smaller, it is necessary to take $N$ larger.
On the other hand, the number $l$ should be taken rather small since the value of 
the first 
summation increases rapidly with $l$ and $s$.  Under the machine precision, $l \leq 10$ is appropriate.

Some specific examples are shown below.
Let $s^*$ be the high-precision zero of \eqref{3-1} calculated by \eqref{2-2}, and let $s^{l,N}$ be the zero of \eqref{4-1} with 
$l$ and $N$. The first two examples are about zeros which will be shown in Figure \ref{X5-2}. 
Concerning $s^* = (0.719846\cdots) + i \; (42.458519\cdots)$, we obtain $s^{l,N}$ 
such as\footnote{Here and in what follows, the symbol $O(10^{-n})$ implies
$\leq C10^{-n}$ with a constant $C>0$.   What we actually want to claim is that
the error is as small as $10^{-n}$, so the numerical value of $C$ is $\lesssim 10$.}
$$
\begin{array}{lll}
|s^*-s^{10,100}| &= |7.8\times10^{-17} + i \; 4.3\times10^{-17}|  &= O\left(10^{-17}\right), \\
|s^*-s^{10,200}| &= |-6.2\times10^{-22} + i \; 3.2\times10^{-22}| &=O\left(10^{-22}\right),
\end{array}
$$
where $s^{10,100}$ is a zero of \eqref{4-1} with $(l,N)=(10,100)$ and so on.
It shows that larger $N$ gives higher accuracy, namely $|s^*-s^{l,N}|$ of $N=200$ is smaller
than that of $N=100$. Similarly, concerning $s^* = (1.043571\cdots) + i \; (98.989673\cdots)$, we have 
$$
\begin{array}{ll}
|s^*-s^{10,100}|  &= O\left(10^{-14}\right), \;\; |s^*-s^{10,200}|  = O\left(10^{-18}\right), \\
|s^*-s^{10,1000}| &= O\left(10^{-26}\right).
\end{array}
$$
Note that the accuracy of the second example is lower than that of the first one with the same $(l, N)$. 
The next example is as to $s^* = (0.778519\cdots) + i \; (799.497864\cdots)$, the imaginary part of which is larger 
than the above two examples.  In this case, we have
$$
\begin{array}{ll}
|s^*-s^{8,200}|  &= O\left(10^{-6}\right), \;\; |s^*-s^{6,300}|  = O\left(10^{-7}\right), \\
|s^*-s^{4,1000}| &= O\left(10^{-12}\right).
\end{array}
$$
For the third example, $l$ should be taken smaller than 10. The reason is that its absolute value is larger than the first two examples.
From these examples, we see that the calculation accuracy gets better by taking $N$ 
large. 
In our actual computations we choose $(l, N)=(10,100)$ for $t < 400$, $(l, N)=(8,200)$ for $400 \le t < 600$ and $(l, N)=(8,300)$ for $600 \le t < 800$, respectively, to realize higher accuracy than $O(10^{-6})$.

Hereafter, we will omit the superscript of $\zeta_2^N$ and $s^{l,N}$ for simplicity.

\subsection{The way to find zeros}\label{subsec-1}

We first search for candidates of zeros by drawing a plot of $|\zeta_2(s,s)|$.   
See Figure \ref{X4}, which shows the absolute value of 
$|\zeta_2(\sigma+it,\sigma+it)|$ for $\sigma=0.56$ and $0 \le t \le 80$.   
In this figure, there is a point which seems to be in contact with the horizontal axis.
Then we use it as a starting value for the root finder {\it FindRoot}, which is a built-in 
function of {\it Mathematica}.  As an option of {\it FindRoot}, we specify WorkingPrecision to 100 digits.

In this way, we draw this type of graphs for various values of $\sigma$ to search for 
candidates of zeros.  Figure \ref{X1-1} is the consequence of such search in the
region $-2 \le \sigma \le 4$ and $0 \le t \le 800$.
In Figure \ref{X1-1}, 
the leftmost zero is $$s = (-0.830372\cdots) + i \; (35.603804\cdots)$$ 
and the rightmost zero is $$s = (1.605277\cdots) + i \; (333.223539\cdots).$$

\begin{figure}
\begin{center} \leavevmode
\includegraphics[width=0.8\textwidth]{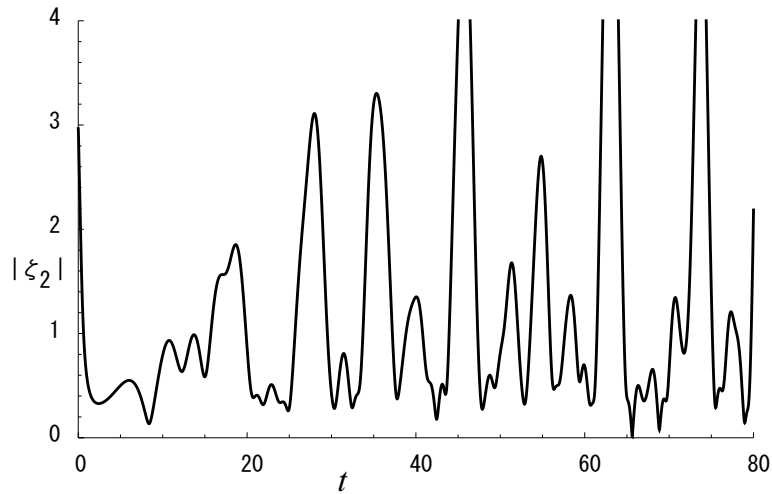}
\end{center}
\caption{A plot of $|\zeta_2|$ when $\sigma=0.56$ to search for a starting value for the root finder.}
\label{X4}
\end{figure}

Some numerical examples of zeros with small imaginary parts are listed below. 
\par
\begin{tabbing}
\hspace{3cm}\= \kill
\> $ ( 0.27672860\cdots) + i ( 8.39755368\cdots) $ \\
\> $ (-0.18995147\cdots) + i (12.30422130\cdots) $ \\
\> $ ( 0.06443907\cdots) + i (15.02312694\cdots) $ \\
\> $ (-0.53767831\cdots) + i (17.58063303\cdots) $ \\
\> $ ( 0.12844956\cdots) + i (20.59707674\cdots) $ \\
\> $ ( 0.08804454\cdots) + i (21.93232180\cdots) $ \\
\> $ ( 1.10778631\cdots) + i (23.79708697\cdots) $ \\
\> $ ( 0.27268471\cdots) + i (24.93425087\cdots) $ \\
\> $ (-0.67413685\cdots) + i (26.88584448\cdots) $ \\
\> $ (-0.15708737\cdots) + i (30.02450294\cdots) $ \\
\> $ ( 0.24085861\cdots) + i (30.35443945\cdots) $ \\
\> $ ( 0.27943393\cdots) + i (32.43085844\cdots) $ \\
\> $ ( 0.19640810\cdots) + i (33.30504691\cdots) $ \\
\> $ (-0.83037218\cdots) + i (35.60380497\cdots) $ \\
\> $ ( 0.26817981\cdots) + i (37.74099414\cdots) $ \\
\> $ ( 1.48543370\cdots) + i (38.13262119\cdots) $ \\
\> $ (-0.45570264\cdots) + i (39.63195833\cdots) $ \\
\> $ ( 0.09633802\cdots) + i (41.36138867\cdots) $ \\
\> $ ( 0.71984635\cdots) + i (42.45851912\cdots) $ \\
\> $ ( 0.32260735\cdots) + i (43.57397755\cdots) $ \\
\> $ ( 0.30547044\cdots) + i (47.82257631\cdots) $ \\
\> $ (-0.07836730\cdots) + i (47.93661087\cdots) $ \\
\> $ ( 0.17623000\cdots) + i (49.35798458\cdots) $ \\
\> $ (-0.10065156\cdots) + i (50.42344359\cdots) $ \\
\> $ ( 1.20851184\cdots) + i (52.67628393\cdots) $ \\
\> $ ( 0.25607674\cdots) + i (52.90185286\cdots) $ \\
\> $ ( 0.26312128\cdots) + i (56.23680524\cdots) $ \\
\> $ ( 0.99787597\cdots) + i (57.00712796\cdots) $ \\
\> $ (-0.54056567\cdots) + i (57.89377726\cdots) $ \\
\> $ ( 0.25852514\cdots) + i (59.37031354\cdots) $ \\
\end{tabbing}

The imaginary parts of the zeros in this list are less than 60, so the calculation
accuracy is much better than $O(10^{-6})$; it is around $O(10^{-17})$.


\subsection{The ``throwing a net and catching fish'' method}
\label{subsec-2}

By the method in the preceding subsection, we can find candidates of zeros,
but we canot determine by that method whether they are really zeros, or 
they are just very small absolute values but not 0.

Therefore, to make sure that they are really zeros, we have to develop another
method.   Let $s^*=\sigma^*+it^*$ be a candidate of zero which we found by the
method in Subsection \ref{subsec-1}.   Consider a small rectangle
$$
R=\{s=\sigma+it\;|\;a\leq\sigma\leq b,\;c\leq t\leq d\}
$$
which includes $s^*$ as an interior point.   Divide the interval $[a,b]$ as
$a=\sigma_0<\sigma_1<\cdots<\sigma_n=b$, and draw the figure of the curves
$$
\mathcal{K}_j=\{\zeta_2(\sigma_j+it,\sigma_j+it)\;|\;c\leq t\leq d\}
\qquad{\rm (}0\leq j\leq n{\rm )}
$$
on the complex plane.   If $R$ is sufficiently small, then the points inside $R$ 
are very close to $s^*$, and hence the curves $\mathcal{K}_j$ locate near the 
origin.   When $j$ moves from $0$ to $n$, the curves $\mathcal{K}_j$ also
move little by little.   If in the course of this moving process $\mathcal{K}_j$
crosses the origin, we should conclude that $\zeta_2(s,s)$ has a zero 
here, because $\zeta_2(s,s)$ is a continuous function (Figure \ref{X5-1} and Figure \ref{X5-2}).   
Therefore $s^*$ should be a zero of $\zeta_2(s,s)$.

\begin{figure}
\begin{center} \leavevmode
\includegraphics[width=0.45\textwidth]{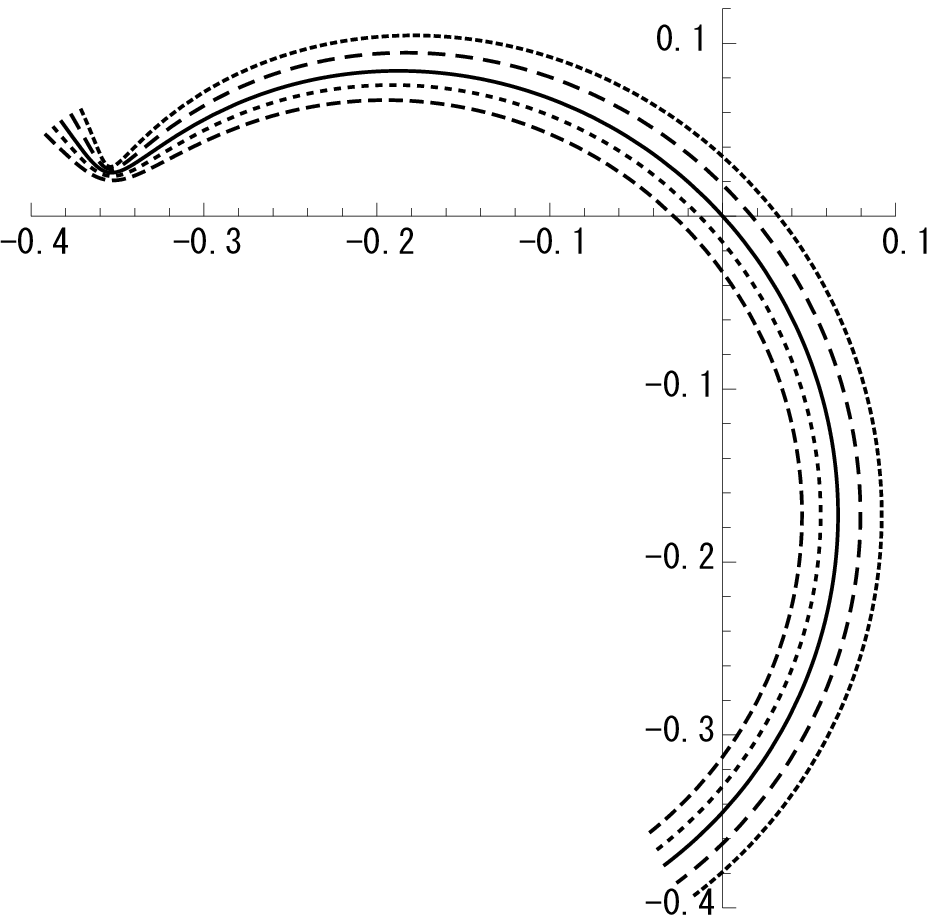} \hskip 1.5cm
\includegraphics[width=0.2\textwidth]{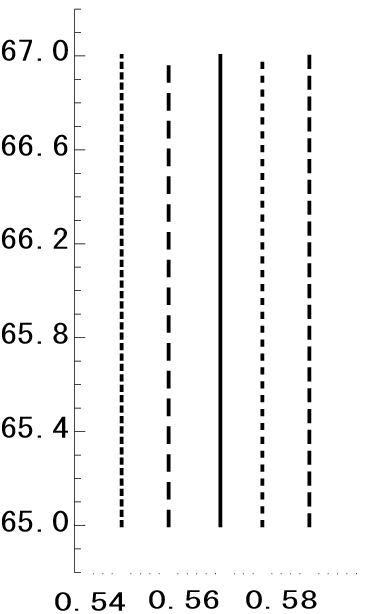} \\ \vskip 0.5cm
$s^* = (0.561016\cdots) + i\; (65.626461\cdots)$
\end{center}
\caption{The left figure is a plot of the curves $\zeta_2(s, s)$ with $65 \leq t \leq 67$ drawn 
in the complex plane.  The point $\zeta_2(s^*,s^*)$ is on the solid line.
The right figure is a plot of $s=\sigma+i t$ in $\sigma$-$t$ plane, 
whose each segment corresponds to each curve on 
the left figure.   The point $s^*$ is on the solid line.
The curves represent the behaviour of $\zeta_2(s, s)$ when 
$\sigma = 0.54, 0.55, 0.561016, 0.57$ and $0.58$, respectively.}
\label{X5-1}
\end{figure}

\begin{figure}
\begin{center} \leavevmode
\includegraphics[width=0.45\textwidth]{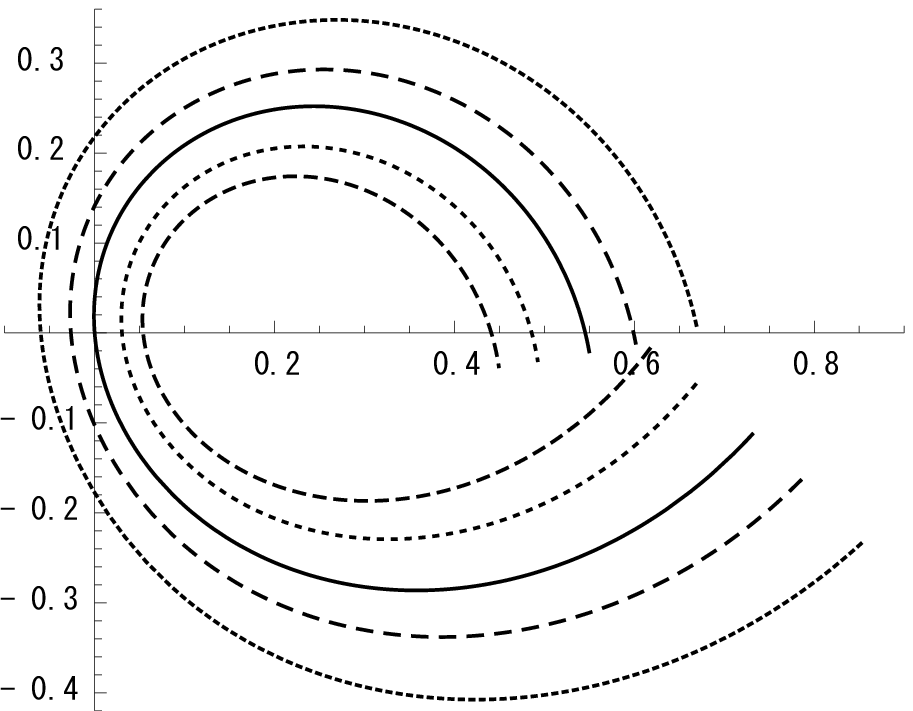} \hskip 1cm
\includegraphics[width=0.45\textwidth]{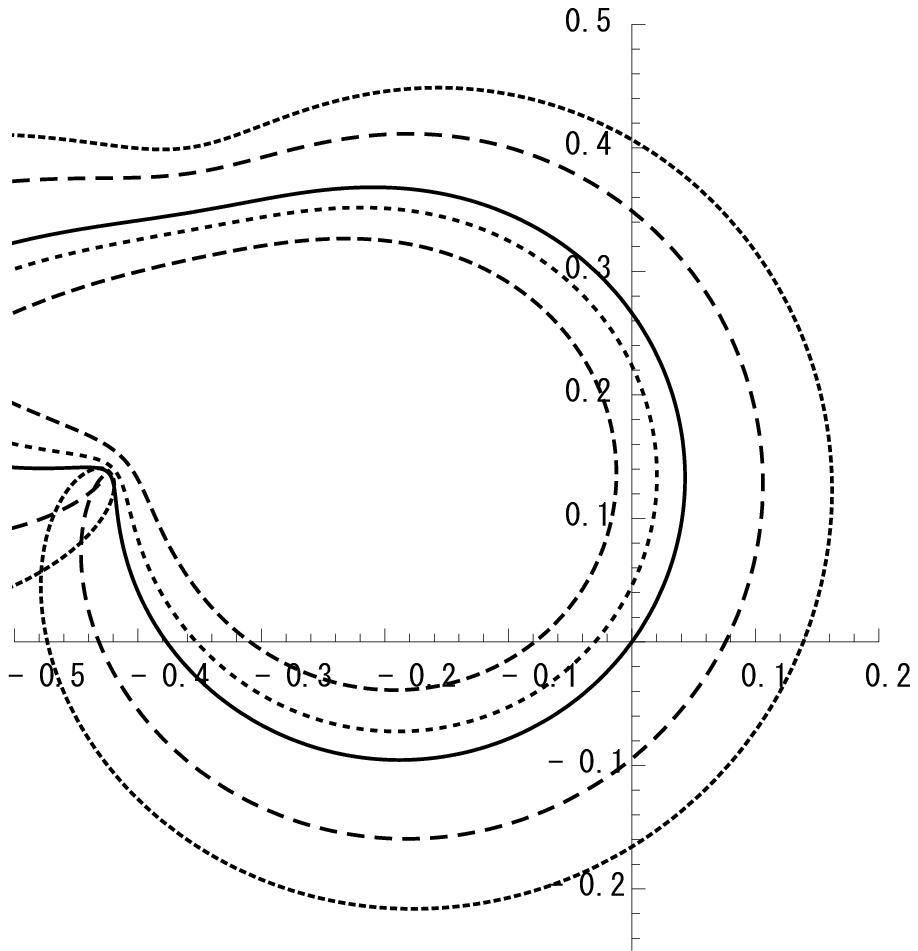} \\ \vskip 0.5cm
(i) $s^* = (1.043571\cdots) + i \; (98.989673\cdots)$ 
\par (ii) $s^* = (0.719846\cdots) + i \; (42.458519\cdots)$
\end{center}
\caption{Plots of the curve $\zeta_2(\sigma+it,\sigma+it)$ drawn in the complex plane. 
The left figure shows the curves with $98 \leq t \leq 100$ for $\sigma = 0.95,\; 1,\; 1.043572,\;
 1.1$ and $1.15$. 
The right figure shows the curves with $41 \leq t \leq 44$ for $\sigma = 0.6, 0.65,\; 0.719846,\;
 0.75$ and $0.8$.}  
\label{X5-2}
\end{figure}

It is to be noted that we have not yet checked all candidates of zeros in Figure
\ref{X1-1} by this  ``throwing a net and catching fish'' method, but we believe that
all of those candidates are indeed zeros.

\subsection{The order of the zeros}

Some zeros in Figure \ref{X1-1} seem to overlap or be very close to each other,
but this is because the vertical scale of Figure \ref{X1-1} is heavily reduced.
All of them are actually isolated from other zeros.
In fact, the closest zeros in the figure are \par
$$(-0.024589\cdots) + i\; (575.888143\cdots)$$
and
$$(0.176317\cdots) + i\; (575.841132\cdots),$$
the distance of which is about $0.206$.

As a precaution, for all zeros here, we checked that the values of the derivative 
$\zeta_2^{\prime}(s,s)$ at those points are non-zero.
The values of the derivative of $\zeta_2(s,s)$ can be calculated, by using
\eqref{2-2}, by a built-in function of {\it Mathematica}.
From those facts, we are sure they are not zero of order two (or more)
\footnote{One might be worried about the possibility that a zero point on
Figure \ref{X1-1} actually represents two zeros which are so close to each other
that cannot be distinguished by our present computations.   However such a possibility
can also be removed by checking the values of derivatives.}. 

\bigskip

\textbf{Acknowledgements.}$\;$   
The authors express their thanks to Professor Aleksandar Ivi{\'c}, 
Professor Ken Kamano and Professor Gediminas Stepanauskas for valuable comments.

\

\end{document}